\documentclass{amsart}   
\usepackage{graphicx}
\usepackage{amsfonts}
\usepackage{url}
\usepackage{hyperref} 
\hypersetup{backref,pdfpagemode=FullScreen,colorlinks=true}
\usepackage{amsmath}
\usepackage{amssymb}
\usepackage{amscd}
\usepackage{color}
\usepackage{amsthm}
\usepackage{bm}
\usepackage{indentfirst}
\usepackage[hmargin=3cm,vmargin=3cm]{geometry}
\usepackage[all, cmtip]{xy}
\numberwithin{equation}{section}
 
\newtheorem{theorem}[equation]{Theorem} 
 
\newtheorem{proposition}[equation]{Proposition}
 
\newtheorem{lemma}[equation]{Lemma} 
 
\newtheorem{corollary}[equation]{Corollary} 
 
\newtheorem{conjecture}{Conjecture}

\newtheorem{definition}[equation]{Definition}
\theoremstyle{definition}

\theoremstyle{remark}

\newtheorem{remark}[equation]{Remark}

\DeclareMathOperator {\sign} {sign}
\DeclareMathOperator {\Id} {Id}

\DeclareMathOperator {\mult} {mult}

\DeclareMathOperator {\Det} {Det}

\begin{document}
\title{Extended Fuller index, sky catastrophes and the Seifert conjecture}
\author{Yasha Savelyev}
\email{yasha.savelyev@gmail.com}
\address{University of Colima, CUICBAS}
\keywords{Fuller index, sky catastrophe, Seifert conjecture, Reeb vector fields}
\subjclass {54H20}
\begin{abstract}  We 
extend the classical Fuller index, and use this to prove that for a certain
general class of vector fields $X$ on a compact smooth manifold, if a homotopy
of smooth non-singular vector fields starting at $X$ has no sky
catastrophes as defined by the paper, then the time 1 limit of the
homotopy has
periodic orbits. This class of vector fields includes the Hopf vector field on
$S ^{2k+1} $. A sky catastrophe, is a kind of bifurcation originally discovered
by Fuller. This answers a natural question that existed since the time of
Fuller's foundational papers. 
 We also put strong constraints on the kind of
 sky-catastrophes that may appear for homotopies of Reeb vector fields.
   \end{abstract}
 \maketitle
\section{Introduction}
The original Seifert conjecture  \cite{citeSeifert} asked if a non-singular vector field on $S
^{3} $ must have a periodic orbit. 
In this formulation the answer was shown to
be no for $C
^{1} $ vector fields by
Schweitzer \cite{citeSchweitzerC1Counterexample}, for $C ^{2} $ vector fields by
Harrison \cite{citeHarrison} and later for $C ^{\infty} $ vector fields by Kuperberg \cite{citeKKuperbergSmoothCounterexample}. A $C^1$ volume preserving counter-example is given by Kuperberg in \cite{citeKuperbergvolumepreserving}. The Hamiltonian analogue was given a counterexample in Ginzburg and G\"urel \cite{citeGinzburgHamiltonianSeifert}.
For a vector field $X$ $C^0$ close to the Hopf vector field it was
shown to hold by Seifert and later by Fuller 
\cite{citeFullerIndex} in his 1967 paper, using his Fuller index. Part of the importance of the  $C
^{0} $ condition for Fuller is
that it rules out ``sky catastrophes'' for an appropriate
homotopy of non-singular vector
fields connecting $X$ to the Hopf vector field. The latter ``sky catastrophes'' 
are the last discovered kind of bifurcations
originally constructed by Fuller himself \cite{citeFullerBlueSky}. 
He constructs a smooth family $\{X _{t} \}, t \in [0,1] $ of vector fields
on a solid torus, 
for which there is a continuous (and isolated) family of $\{X _{t}\} $ periodic orbits
$\{o _{t}\} $,  with the period of $o _{t} $ going
to infinity as $t \mapsto 1$, and so that for $t =1$ the orbit
disappears. We can make the following slightly more general preliminary definition.   
\begin{definition} [Preliminary] 
  A \textbf{\emph{sky catastrophe}} for a smooth family $\{X _{t} \}$, $t \in [0,1]$,
   of vector
   fields on a manifold $M$ is a continuous family of closed orbits $\tau \mapsto o _{t _{\tau} }$, $o _{t _{\tau} } $ is a non-constant periodic orbit of $X _{t _{\tau} } $, $\tau \in [0, \infty)$, such that the period of $o _{t _{\tau} } $ unbounded from above.
\end{definition}
These sky catastrophes (and their more robust analogues called blue sky catastrophes) turned out to be common in many kinds of systems
appearing in nature and have been studied on their own, see for instance
Shilnikov-Turaev~\cite{citeShilnikovTuraevBlueSky}.

However since the time of Fuller's original papers it has not been
understood if this the only thing that can go wrong. That is if without 
existence of a ``sky catastrophe'' in an appropriate general sense, the time 1
limit of a homotopy of smooth non-singular vector fields on $S ^{2n+1} $
starting at the Hopf vector field must have a periodic orbit.
The difficulty in answering this is that
although our orbits cannot ``disappear into the sky'', as there are infinitely
many of them, they may ``cancel each other out'' even if the Fuller index is
``locally positive'' - that is the index of isolated components in the orbit space is positive.
In the $C ^{0} $ nearby case this cancellation is prevented as orbits from isolated components of the orbit space may not
interact.
The reader may think of trying to make sense of the
infinite sum $$ (5-1) + (5-1) + \ldots + (5-1) + \ldots.
$$ 
While generally meaningless it has some meaning if we are not allowed to move
the terms out of the parentheses.
So one has to develop a version
of Fuller's index which precludes such total cancellation in general.
This is what we do here and using this answer affirmatively the above question, and for more general kinds of manifolds and vector fields. 

First we
define our general notion of a ``sky catastrophe''.
Given a homotopy of smooth vector fields $\{X _{t}
\}$ on $M$ we define the space of non-constant periodic orbits of $\{X _{t} \}$: 
\begin{multline} \label{eqSfirst}
S = S (\{X _{t}
\}) = \{(o, p, t) \in LM \times (0, \infty) \times [0,1]
\,\left .  \right |
\, \\ \text{ $o: \mathbb{R}/\mathbb{Z} \to M $  is a
non-constant periodic orbit of $\frac{1}{p} X _{t} $} \}.
\end{multline}
Here $LM$ denotes the free loop space.
We have an embedding 
$$emb: S \hookrightarrow M \times (0,\infty)
\times [0,1]$$ given by $(o, p, t) \mapsto (o (0), p, t) $ and $S$ is given the
corresponding subspace
topology. Further on the same kind of topology will be used on related spaces.
(It is the same as the induced topology from compact-open or Frechet
topology on $LM$.) 
\begin{definition}
Suppose that $\{X _{t} \}$, $t \in [0,1]$ is a smooth family of vector fields on
   $M$.
   We shall say that $\{X _{t} \}$ has a \textbf{\emph{sky catastrophe}}, if
   there is an element $$y \in S \cap \left (LM  \times (0, \infty) \times \{0\} \right ),$$  so
that there is no open-closed and bounded from above subset of $S$ containing
$y$, where bounded from above means
that the projection to the component $(0,\infty)$ is bounded from above.
\end{definition}
Then clearly Fuller's catastrophe is a special case of the above definition.
\begin{remark}
The above general definition coincides with the preliminary definition if $M$ is
compact and the
connected components of $S
(\{X _{t} \})$ are open and path-connected, which likely happens generically.
\end{remark}
We point out that when $M$ is compact a subset of $S$ is compact if and only if it is closed and its
projection to $(0,\infty)$ has a non-zero bound from below.
This of course is immediate by existence of the topological embedding $emb$ above.
Thus when $\{X _{t} \}$ is a homotopy of non-singular vector fields, and $M$ is
compact,
we may replace the open-closed and bounded from above condition by the more
technically useful open and compact
condition.

\begin{theorem} \label{conj:preliminary} Let $X = X _{1} $ be a smooth non-singular vector field on $S ^{2k+1} $ homotopic to
the Hopf vector field $H = X _{0} $ through  homotopy  $\{X _{t} \}$ of smooth non-singular vector fields.
Suppose that $\{X _{t} \}$ has no sky catastrophes
then $X$ has periodic orbits.
\end{theorem}
Let us call a homotopy $\{X _{t} \}$ satisfying the conditions of the theorem
above \emph{partially admissible}. 
\begin{lemma} \label{lemmaExample} There exists a $\delta >0$ so that
whenever $X$ is $C ^{0} $ $\delta$-close to the Hopf vector field $H$, $X = X
_{1} $ for a partially
 admissible homotopy $\{X _{t} \}$, with $X _{0}= H$. And in particular by the above $X$
has periodic orbits.
\end{lemma}
Thus Theorem \ref{conj:preliminary}  may be understood as an extensive
generalization of the theorem of Seifert giving existence of periodic orbits for
non-singular vector fields $C^0$ close to the Hopf vector field, on which the
Seifert conjecture was based.

It is interesting to
consider the 
contrapositive of the theorem above, and which may be understood as a new phenomenon and one concrete application of our theory.
\begin{corollary} \label{cor:obvious} Given any homotopy $\{X _{t} \}$ of smooth
non-singular vector fields from the
Hopf vector field to a vector field with no periodic orbits,  $\{X _{t} \}$ has a sky catastrophe.
\end{corollary}
Note that by the main construction in Wilson~\cite{citeWilsonOnTheminimalsets}, the
Hopf vector field $H$ is homotopic through smooth non-singular vector fields to a vector
field with finitely many simple closed orbits. Combining this with the construction in \cite{citeKKuperbergSmoothCounterexample}
we find that there do exist  homotopies of $H$ through smooth non-singular vector
fields to a vector field
with no closed orbits. 

More general and extended forms of the above theorem are stated in Section
\ref{section:generalization}. 
To prove them we give a certain natural extension of the classical Fuller
index, with the latter giving  certain invariant rational counts of periodic orbits of a smooth vector
field in ``dynamically isolated compact sets''.
Our extension is $\mathbb{Q} \sqcup \{\pm \infty\}$ valued. One ingredient for this is a notion of perturbation system for a vector
field, which will allow us to consider weighted in terms of index and
multiplicity ``infinite sums'' of closed orbits of a vector field. For these
sums to have any meaning we impose certain ``positivity'' or ``negativity''
conditions. 
\begin{remark}
These kind of summations are used in ``positive topological quantum
field theories
'' where infinite sums which are normally meaningless with coefficients in a
ring like $\mathbb {C}$ are made meaningful by working with a complete
semi-ring in the sense of Samuel Eilenberg. But what we do here is of course much
more basic in principle.
\end{remark}

Can we use Theorem \ref{conj:preliminary} and its analogues for more general
manifolds in Section
\ref{section:generalization} to show existence of orbits? Let us assume some
minimal regularity on the homotopy $\{X _{t} \}$ so that connected components of
$S (\{X _{t} \})$ are open.
Then ideally we would like to have some a priori upper bounds for the
period on connected components coming from some geometry-topology of the
manifold and or the vector field. This then means that connected components of
$S (\{X _{t} \})$ would be open and compact, and we can apply our theorems to get existence
results.
Fuller himself in
\cite{citeFullerIndex} gives an example of such bounds for vector fields on the
2-torus, in fact his bounds are absolute for the whole $S$ (for a fixed homotopy
class of orbits) not just its
connected components. He also speculates that if one works with divergence free vector fields
one can do more to this effect. Below we examine the case of Reeb vector fields.
\subsection {Reeb vector fields and sky catastrophes}
\label{section:Reebadmissible}
Viterbo~\cite{citeViterboWeinstein} shows that Reeb vector fields for the standard contact
structure on $S ^{2k+1} $ always have periodic orbits. In
the case of $S ^{3} $ for an overtwisted contact structure
this is shown by Hofer
\cite{citeHoferWeinsteinSymplectizationInvetiones},
using pseudo-holomorphic curve techniques. 
In the context of the Seifert
conjecture, it is interesting  to understand 
 the most basic properties of qualitative-dynamical, or even just topological
 character,
 that Reeb vector fields posses, which makes them
different from general non-singular vector fields. Indeed as a first step, in
light of our results, we may
ask if a homotopy of Reeb vector fields $\{X _{t}
\}$ on a closed manifold is necessarily free of sky catastrophes.  The  following theorem puts
a very strong restriction on the kinds of sky catastrophes that can happen. 
It is likely, that if they exist, they must be pathological, and very hard to construct.
\begin{remark}
See however Kerman \cite[Theorem 1.19]{citeKerman2017}, where a kind of partial Reeb plug is constructed, which is missing the matching condition, see for instance \cite{citeKKuperbergSmoothCounterexample} for terminology of plugs, also see Kerman \cite{citeKermanHamiltonianSeifert}, Ginzburg and G\"urel \cite{citeGinzburgHamiltonianSeifert}, where plugs are utilized in Hamiltonian context.
If one had a plug with all conditions, then it is simple to construct a sky catastrophe. For we may deform such a plug through partial plugs satisfying all conditions except the trapping condition (condition 3 in \cite{citeKKuperbergSmoothCounterexample}) to a trivial plug. This deformation then readily gives a sky catastrophe corresponding to the trapped orbit.
Without matching, this argument does not obviously work.
\end{remark}

The proof of the following uses only elementary geometry of the
Reeb vector fields.  

%
\begin{theorem} \label{prop:AdmissibleReeb} Let $\{X _{t} \}$, $t \in [0,1]$ be a smooth homotopy through Reeb
vector fields on a contact manifold $M$.
   Let $S = S (\{X _{t} \})$ be defined as before, then there is no unbounded from above locally Lipschitz  continuous $p: [0,\infty)
\to S$ whose composition with the projection $\pi_3: LM
\times (0, \infty) \times [0,1] \to [0,1]$ has finite length. 


\end{theorem}
The above rules out for example Fuller's sky catastrophe that appears in
\cite{citeFullerBlueSky}, and described in the beginning of our paper.
\begin{conjecture} Let $\{X _{t} \}$, $t \in [0,1]$ be a smooth homotopy through Reeb
vector fields on a compact contact manifold $M$. Then there is a $C ^{0} $ nearby smooth family $\{X' _{t} \}$, $t \in [0,1]$, $X' _{i}=X _{i}  $, $i=0,1$, such that $\{X' _{t} \}$ has no sky catastrophes.
\end{conjecture}
Given this conjecture we may readily apply (general analogues of) Theorem \ref{conj:preliminary} to get applications to existence of Reeb orbits.
\section {Fuller index and its extension}
 \label{sec:Fuller}
The Fuller index is an analogue for orbits of the fixed point index,
but with a couple of new ingredients: we must account for the
symmetry groups of the orbits, 
and since the period is freely varying there is
an extra compactness issue to deal with. 
Let us briefly recall the definition following Fuller's original
paper \cite{citeFullerIndex}. All vector vector fields from now on, everywhere
in the paper, will be assumed to be
smooth and non-singular and manifolds smooth and  oriented (the last for
simplicity).

Let $X$ be a vector field  on $M$. Set
\begin{equation*}
S (X) = S (X, \beta) = 
   \{(o, p) \in L _{\beta} M \times (0, \infty) \,\left .  \right | \, \text{ $o: \mathbb{R}/\mathbb{Z} \to M $ is a
   periodic orbit of $\frac{1}{p} X $} \},
\end{equation*}
where $L _{\beta} M  $ denotes
the free homotopy class $\beta$ component of the free loop space.
Elements of $S (X)$ will be called orbits. There
is a natural $S ^{1}$ reparametrization action on $S (X)$, and elements of $S
(X)/S ^{1} $ will be called \emph{unparametrized orbits}, or just orbits. Slightly abusing notation we write $(o,p)$  for
the equivalence class of $(o,p)$. 
%
%
The multiplicity $m (o,p)$ of a periodic orbit is
the ratio $p/l$ for $l>0$ the least  period of $o$.
We want a kind of fixed point index which counts orbits
$(o,p)$ with certain weights - however in general to get
invariance we must
have period bounds. This is due to potential existence of sky catastrophes as
described in the introduction.

Let $N \subset S (X) $ be a compact open set (the open condition is the meaning
in this case of
``dynamically isolated'' from before).
Assume for simplicity that elements $(o,p) \in N$  are
isolated. (Otherwise we need to perturb.) Then
to such an $(N,X, \beta)$
Fuller associates an index: 
\begin{equation*}
   i (N,X, \beta) = \sum _{(o,p) \in N/ S ^{1}}  \frac{1}{m
   (o,p)} i (o,p),
\end{equation*}
 where $i (o,p)$ is the fixed point index of the time $p$ return
map of the flow of $X$ with respect to
a local surface of section in $M$ transverse to the image of $o$.
Fuller then shows that $i (N _{t} , X _{t}, \beta )$ is invariant for a deformation $\{X _{t}
\}$ of $X$ if $N _{t} $ is dynamically isolated for all $t$, that is if $\bigcup _{t} N _{t}  $ is 
open in $S (\{X _{t} \})$.

In the case where $X$ is
the $R ^{\lambda} $-Reeb vector field on a contact manifold $(C ^{2n+1} , \xi)$,
and if $(o,p)$ is
non-degenerate, we have: 
\begin{equation} \label{eq:conleyzenhnder}
i (o,p) = \sign \Det (\Id|
   _{\xi (x)}  - F _{p, *}
^{\lambda}| _{\xi (x)}   ) = (-1)^{CZ (o)-n},
\end{equation}
where $F _{p, *}
    ^{\lambda}$ is the differential at $x$ of the time $p$ flow map of $R ^{\lambda} $,
    and where $CZ ( o)$ is the Conley-Zehnder index, (which is a special
    kind of Maslov index) see
    \cite{citeRobbinSalamonTheMaslovindexforpaths.}.

\subsection{Extending Fuller index}
\label{sub:extendingfullerindex}
We assume from now on, everywhere in the paper, that $M$ is compact, smooth, and
oriented, although we sometimes reiterate for clarity.
We will now describe an extension of the Fuller index  allowing us to work with the entire
Fuller phase space $LM \times \mathbb{R} _{+} $. We define an  index $i
(X, \beta)$ that
depends only on $X, \beta$ but which is a priori defined and is invariant only for certain
special vector fields and 
homotopies of vector fields.  
In fact we have already given the essence of the
necessary condition  on 
the homotopy in the special case of Theorem \ref{conj:preliminary}. 
Let
\begin{equation*}
 S (\{X _{t} 
\}, \beta) = \{(o, p, t) \in L _{\beta} M \times (0, \infty) \times [0,1]
\,\left .  \right |
\, \text{ $o: \mathbb{R}/\mathbb{Z} \to M $ is a
periodic orbit of $\frac{1}{p} X _{t} $} \}.
\end{equation*}
\begin{definition} 
For a smooth homotopy $\{X _{t} \}$ of smooth non-singular vector fields on $M$, we say
that it is \textbf{\emph{partially admissible in free homotopy class $\beta$}},
if every element of
  $$ S (\{X _{t} \}, \beta) \cap \left (L _{\beta} M  \times (0, \infty) \times
   \{0\} \right )$$
is contained in a compact open subset of $S (\{X _{t} \}, \beta)$.
 We say that $\{X _{t} \}$ is \textbf{\emph{admissible in free homotopy class $\beta$}} if
every element of
  $$S (\{X _{t} \}, \beta) \cap \left (L _{\beta} M  \times (0, \infty) \times
   \partial [0,1] \right )$$
is contained in a compact open subset of $S (\{X _{t} \}, \beta)$.
\end{definition}

For $X$ a vector field, we set 
\begin{equation}
\begin{aligned}
S   (X )  & = \{(o, p) \in L M   \, \left .  \right | \, \text{ $o: \mathbb{R}/\mathbb{Z} \to M $ is a
periodic orbit of $\frac{1}{p} X $}  \}. \\
   S   (X,\beta )  & = \{(o, p) \in L _{\beta}M   \, \left .  \right | \, \text{ $o: \mathbb{R}/\mathbb{Z} \to M $ is a
periodic orbit of $\frac{1}{p} X $}  \}. \\
   S   (X,a, \beta )  & = \{(o, p) \in S (X, \beta)  \, \left .  \right | \, p
   \leq a \}. \\
\end{aligned} 
\end{equation}
\begin{definition} \label{def:perturbedsystem}
Suppose that $S   (X, \beta )$ has open connected components.
And suppose that we have a collection of 
vector fields $\{X ^{a} \} $, for each $a>0$,
satisfying the following:
\begin{itemize}
   \item 
  $S   (X ^{a},a, \beta ) $ consists of isolated orbits for each $a$.
\item $$S   (X ^{a},a, \beta ) = S   (X ^{b},a, \beta ),  $$ (equality of
subsets of $L _{\beta} M \times \mathbb{R} _{+} $) if $b>a$, and
the index and multiplicity of the orbits corresponding to the identified elements of these sets
coincide.
\item
There is a prescribed homotopy $\{X ^{a}_t \}$ of each $X ^{a} $ 
to $X$, 
called \textbf{\emph{structure
homotopy}}, 
with the property that for every $$y \in S  (\{X ^{a}_t \})
\cap \left ( L _{\beta} M   \times (0,a] \times \partial [0,1] \right)$$
there is an open
compact subset $\mathcal{C} _{y} \ni y  $ of $S  (\{X ^{a}_t \})$
   which is 
       \textbf{\emph{non-branching}} which means that 
$$\mathcal{C} _{y}  \cap \left (L _{\beta} M  \times (0, \infty) \times
\{i\} \right ),$$ $i=0,1$ are connected.
     \item $$S   (\{X ^{a} _{t}  \}, \beta ) \cap \left (L _{\beta} M \times (0,a] \times
        [0,1] \right) = S   (\{X ^{b} _{t}  \},
        \beta ) \cap \left ( L _{\beta} M \times (0,a] \times [0,1] \right),  $$ (equality of
         subsets of $L _{\beta} M \times \mathbb{R} _{+}  \times [0,1])$ if $b>a$ is
         sufficiently large.
\end{itemize}
We will then say that $\{X ^{a}  \}  $ is a 
\textbf{\emph{perturbation system}} for $X$ in the class $\beta$, (keeping track
   of structure homotopies and of $\beta$ implicitly).
\end{definition}


We shall see shortly that a Morse-Bott Reeb vector field always admits a perturbation
system. 
\begin{definition} \label{def:infinitetype}  Suppose that  
$X$ admits a perturbation system $\{X ^{a} \}$ so that there exists an $E = E
   (\{X ^{a} \})$ with
the property that $$S  (X
^{a}, a, \beta  ) = S  (X ^{E}, a, \beta   )$$ for all $a > E$, where this as
before is equality of subsets of $L _{\beta} M \times \mathbb{R} _{+} $, and the
index and multiplicity of identified elements are also identified.
Then we say that $X$ is \textbf{\emph{finite type}} and set:
   \begin{equation*}
      i (X, \beta) = \sum _{(o,p) \in S   (X ^{E}, \beta )/S ^{1} }  \frac{1}{m
   (o,p)} i (o,p).
\end{equation*}

\end {definition}
  \begin {definition} 
   Otherwise, suppose that $X$ admits a perturbation system $\{X ^{a}   \}
    $ and 
   there is an $E = E (\{X ^{a} \}) >0$ so
   that the index $i (o,p)$ of elements $(o,p) \in S(X ^{a}, \beta) $ with $E \leq
   p \leq a $ is positive, respectively negative for every $a >E$,
and s.t. 
\begin{equation*}
\lim _{a \mapsto \infty} \sum _{(o,p) \in S (X,a,\beta) /S ^{1}}  \frac{1}{m
   (o,p)} i (o,p) =  \infty, \text{ respectively } \lim _{a \mapsto \infty} \sum _{o \in S (X,a,\beta) /S ^{1}}  \frac{1}{m
   (o,p)} i (o,p) = - \infty.
\end{equation*}
   Then
   we say that $X$ is \textbf{\emph{positive infinite type}}, respectively \textbf{\emph{negative
   infinite type}} and set $i
   (X, \beta)= \infty, $ respectively
   $ i (X, \beta) = - \infty $.
   We say it is
   \textbf{\emph{infinite type}} if it is one or the other.  
\end{definition}   
\begin{definition} We say that $X$ is
   \textbf{\emph{definite}} type if it is infinite type or finite type.
\end{definition}
With the above definitions $$i (X, \beta) \in \mathbb{Q} \sqcup {\infty} \sqcup {- \infty}, $$
when it is defined.
\begin{remark}
  It is an elementary exercise that the condition that 
   \begin{equation*}
   \lim _{a \mapsto \infty} \sum _{o \in S   (X ^{a}, a, \beta )/S ^{1}}  \frac{1}{m
   (o)} i (o) =  \infty, \text{ respectively } \lim _{a \mapsto \infty} \sum _{o \in S   (X ^{a}, a, \beta )/S ^{1}}  \frac{1}{m
   (o)} i (o) = - \infty,
   \end{equation*} actually follows if the other conditions are satisfied,
   so is only stated for emphasis.
\end{remark}

\begin{definition} A vector field $X$ is \textbf{\emph{admissible}} if 
   it admits a
perturbation system, and if it is definite type.
\end{definition}
\begin{remark}
One may be tempted to extend the finite type to include the case when   
$$\lim _{a \mapsto \infty} \sum _{o \in S   (X ^{a}, a, \beta )/S ^{1}}  \frac{1}{m
   (o)} i (o), $$ exists and the associated series is absolutely convergent. This
   definitely works for our arguments later, but it is an elementary exercise  to show that
   there are no such $X$ unless it is of the previous finite type.
\end{remark}
\subsubsection {Perturbation systems for Morse-Bott Reeb vector fields}
\begin{definition} \label{def:MorseBott}
A contact form $\lambda$ on $M$, and its associated flow $R ^{\lambda} $ are called \emph{Morse-Bott} if the $\lambda$
action spectrum $\sigma (\lambda)$ - that is the space of
   critical values of $o \mapsto \int _{S ^{1} } o ^{*} \lambda  $,
is discreet and if for every $a \in
\sigma (\lambda)$, the space $$N _{a}: = \{x \in M | \, F _{a} (x)
=x \},   $$ $F _{a} $ the time $a$ flow map for $R ^{\lambda} $ - is a closed smooth
manifold such that rank $d \lambda |  _{N _{a} } $ is locally constant
and $T _{x} N _{a} = \ker (d F _{a} - I ) _{x}   $.
\end{definition}

\begin{proposition} \label{prop:MorseBott1}
  Let  $\lambda$ be a contact form of Morse-Bott type, on a closed contact
  manifold $C$. Then the corresponding Reeb vector field $R ^{\lambda} $  admits a
  perturbation system $\{X  ^{a }  \}$, for every class $\beta$, with each $X ^{a} $ Reeb so
  that all the structure homotopies are through Reeb vector fields. 
\end{proposition}
The above very likely extends to more general ``Morse-Bott type'' and beyond vector fields of non
Reeb type,  however one must take care to give the right definitions.


\begin{proof}
Let $$ O _{ \leq E}= O _{ \leq E} (R ^{\lambda} ) \simeq S (R ^{\lambda}, E) $$ denote the set
     of points $x \in C$, s.t. $F ^{\lambda}
    _{p} (x)=x $, for $F ^{\lambda} _{p}   $  the time $p \leq E$ flow map for $R
    ^{\lambda} $.  
Given an $a$ take an $E >a$ s.t. the set $O _{\leq E} = O _{\leq E} (R ^{\lambda} )  $
    is a union of closed manifolds (of varying dimension), call such an $E$ \emph{appropriate}. 
Let
   $\mathcal{O}
   _{\leq E} $ be the natural $S ^{1} $-quotient of $O _{\leq E  } $. By
   \cite [Section 2.2]{citeFredericBourgeois} we
   may find a smooth function $f _{E } $ on $C$
   with support in a normal neighborhood of ${O} _{\leq E
   } $, with $Df _{E } (R ^{\lambda} )  =0$ on $O _{\leq E  } $
   descending to a Morse function on the union of closed orbifolds $\mathcal{O}
   _{\leq E  } $. 
   
Let $\lambda _{E, \mu} = (1 + \mu f _{E} )
   \lambda $. By \cite [Section 2.2]{citeFredericBourgeois} we may choose $\mu
   _{0}  > 0$ so that elements of
   $\mathcal{O} _{\leq E}  (R ^{\lambda _{E, \mu}   }) $ are non-degenerate and
   correspond to critical points of $f _{E  } $, for $0 < \mu \leq \mu (E) $.
   Let $\{E _{n} \} $ be an increasing sequence of appropriate
   levels as above. Since the action spectrum of $\lambda$
   is discreet by
   the Morse-Bott assumption,  we may take $\{f _{E _{n} } \} $ so that $f _{E _{n'}
   }| _{E _{n} }  $ coincides with $f _{E _{n} } $ if $E _{n'} > E _{n}  $. 
   (Note however that the cutoff value $\mu (E _{n'})$ needs to in general be
   smaller then $\mu (E _{n} )$.)
   Given this, 
   we set $X  ^{a} = R ^{\lambda _{E _{n}, \mu (E _{n}) }}  $, for any $E _{n}>a $.
   For the structure homotopies we take the obvious homotopies induced by the
   homotopies $$t \mapsto (1+ (1-t) \mu (E _{n} ) f _{E _{n} }  ) \lambda, \quad t \in
   [0,1],$$ of the contact forms.
\end {proof}
\begin{lemma} \label{lemma:HopfinfiniteType} The Hopf vector field $H$ on $S
^{2k+1} $ is infinite type.
\end{lemma}
\begin{proof}  Pick a perfect Morse function $f$ on $\mathbb{CP} ^{k} $. 
This induces a perfect Morse function $f$ on $\mathcal{O} _{ \leq 2 \pi n}  $, 
upon identifying
   $\mathcal{O} _{ \leq 2 \pi n} $  with the $n$-fold disjoint union of copies of
   $\mathbb{CP} ^{k} $ (forgetting the totally non-effective orbifold
   structure).
Use the
   construction 
   above, to obtain  a perturbation system
   $\{H  ^{2 n \pi}  \}$ for $H$,  $n \in
   \mathbb{Z}_+$,   so that the space
   $\mathcal{O} ^{pert} _{\leq 2 \pi n} = \mathcal{O}  _{\leq 2 \pi n} (H ^{2 n \pi})  $,  be identified with critical points
   of $f $ on the space
   $\mathcal{O} _{ \leq 2 \pi n}  $.     Given a critical point $p$ of $f$ on
   the component $\mathcal{O} _{2 \pi i} \simeq \mathbb{CP} ^{k} 
    $, $0 < i \leq  2 \pi n $, of $\mathcal{O} _{\leq 2 \pi n} $, let $o _{p} $
   denote the corresponding orbit in $\mathcal{O} ^{pert} _{2 \pi i}  $.
   By \cite
   [Lemma 2.4]{citeFredericBourgeois}, 
   \begin{equation*} \label{eq:}
  \mu _{CZ} (o_p)   = \mu _{CZ} (\mathcal{O} _{2 \pi i} ) -  \frac{1}{2} \dim
  _{\mathbb{R}} 
  \mathcal{O} _{2 \pi i} + morseindex _{f} (p),
   \end{equation*}
where $\mu _{CZ} (\mathcal{O} _{2 \pi i} ) $ is the generalized Maslov index for
an element of $\mathcal{O} _{2 \pi i}$, see for instance 
\cite [Section 5.2.2]{citeFredericBourgeois}.
 Let us slightly elaborate as our
reader may not be familiar with this. We pick a
representative for a class of an orbit $o$ in $\mathcal{O} _{2 \pi i}$, pick a
bounding disk for the orbit and using this choice trivialize the contact
distribution along $o$. Given this trivialization the $R^{\lambda} $ Reeb flow induces a path of
symplectic matrices to which we apply the generalized Maslov index. 
Since $\mu _{CZ} (\mathcal{O} _{2 \pi i} )$  has even parity for all $k$, it
follows that $\mu _{CZ} (o_p)$
   has the same parity for all $n$, and so by \eqref{eq:conleyzenhnder} $H$ is
   infinite type. 
\end{proof}

\section {Extensions of theorem \ref{conj:preliminary} and their proofs}

\label{section:generalization}
\begin{theorem} \label{thm:welldefined} 
Suppose we have an admissible vector field $X _{0}  $, with $i (X
_{0}, \beta ) \neq 0$  on a closed, oriented
manifold $M$, which is joined to $X _{1} $  by a partially admissible 
homotopy $\{X _{t} \}$, 
 then  $X _{1} $ has periodic orbits. 
\end{theorem}
Theorem \ref{conj:preliminary} clearly follows by the above and by Lemma
\ref{lemma:HopfinfiniteType}. What follows is a more precise result.
\begin{theorem} \label{thm:welldefinedadmissible} If $M$ is closed, oriented  and $X
_{0}, X _{1}  $ and $\{X _{t} \}$ are admissible then $i (X_0, \beta) = i (X _{1},
\beta)$.
\end{theorem}
\subsection {Fuller correspondence} 
  We need a beautiful construction of Fuller \cite{citeFullerIndex}, which converts contractible
orbits of $X$ into non-contractible orbits in an associated space, for an
associated vector field, so that the correspondence between the periodic orbits
is particularly suitable. 
Let $\textbf{M}'$ be the subset of the $k$-fold product $ M \times \ldots \times
M
$, for $k$ a prime, consisting of points all of whose coordinates $(x_1, \ldots , x _{k} )$
are distinct. Let $\textbf{M}$ be the quotient of $\textbf{M}'$ by the
permutation action of $\mathbb{Z} _{k} $, generated by $P (x_1, x_2, \ldots
x_k)= (x_2,  \ldots ,  x _{k}, x _{1}  )$. As this is a free action the
projection 
map $\textbf{M}' \to \textbf{M} $ is a regular $k$-sheeted covering.
And so we have a homomorphism $$\mu: \pi _{1} (\textbf{M}) \to \mathbb{Z}_k, $$
which extends to $H _{1} (\textbf{M}) $ since $\mathbb{Z} _{k} $ is abelian. 

A vector field $X$ on $M$ determines a vector field $\textbf{X}$ on $\textbf{M}$
given by 
\begin{equation} \label{eqbfX}
\textbf{X} (\textbf{x}) = [(X (x_1), \ldots, X (x_k) ],
\end{equation} 
$\textbf{x} = [(x_1,  \ldots, x_k )]$. It is easy to see $\textbf{X}$ is
complete when $X$ is complete, which holds in our case by compactness of $M$.
Now
for an orbit $(o,p) \in S (X, \beta)$ with multiplicity $m < k$ define a
   multiplicity $m$ orbit $ (\textbf{o}, \frac{p}{k})  \in S (\textbf{X})$.

 $$\textbf{o} (t)  = [o (t/k), 
   o (t/k + \frac{1}{k}), \ldots,  o (t/k + (k-1)/k)] \in L\textbf{M}.
$$ 
   Clearly $\mu ([\textbf{o}]) = \textbf{1}$,  ($\textbf{1}$ corresponding to the  generator $P $
of the permutation group $\mathbb{Z} _{k} $, with $P$ as above)
and $i (\textbf{o}) = i (o)$ by Fuller \cite [Lemma 4.5]{citeFullerIndex}.
If we work over all classes $\beta$ then it is easy to see that  
\begin {equation}
\begin{split}
   &   Ful _{k} : (o, p) \mapsto (\textbf{o}, \frac{p}{k}), 
\end{split}
\end {equation}
is a bijection from the set of all (unparametrized) period $p$ orbits of $X$
with multiplicity less than $k$ to the
set of all (unparametrized) period $\frac{p}{k}$ orbits of $\textbf{X}$ with
multiplicity less than $k$.
\begin{lemma} \label{lemmaMultiplicity}
Let $\{X _{t} \}$ be usual and   let $\textbf{m} (\{X _{t} \},a)$ denote the least upper bound for the set of
   multiplicities of elements of $$S   (\{X  _{t}  \}, \beta ) \cap \left (L _{\beta} M \times (0,a] \times
        [0,1] \right)  $$ then
   $\textbf{m} (\{X _{t} \},a) < \infty$.
\end{lemma}
\begin{proof} This is a version of \cite [Lemma 4.2]{citeFullerIndex}. The proof is
   as follows. Since $$S   (\{X  _{t}  \}, \beta ) \cap \left (L _{\beta} M \times (0,a] \times
        [0,1] \right) $$ is compact as it is identified under
   the embedding $emb: S   (\{X  _{t}  \}, \beta ) \to M \times (0, \infty) \times
   [0,1] )$, $(o, p, t) \mapsto (o (0), p, t) $ with a closed bounded subset of a finite dimensional manifold,
   we would otherwise have a convergent sequence $\{(o _{k}, p _{k})\}$ in $S   (\{X  _{t}  \}, \beta )$
   with $\{t _{k} \}$ also convergent, and
   so that $p _{k}$ converges to 0. But this contradicts the assumption that $X
   _{t} $ are non-singular.
\end{proof}
We set 
\begin{equation} \label{eqFulS}
S (\textbf{X}, a, \beta) = Ful _{k} (S (X, a, \beta)),
\end{equation}
for $k> m (X,a)$.
%

\subsection {Preliminaries on admissible homotopies}
\begin{definition} \label{def:admissible}
  Let $\{X _{t} \}$ be a smooth homotopy of non-singular vector fields.  
  For $b >a >0$ we say that $\{X _{t} \}$
  is \textbf{\emph{partially}} $a,b$-\textbf{\emph{admissible}}, respectively
   $a,b$-\textbf{\emph{admissible}} (in class $\beta$)  if 
   for each $$y \in \left(S = S (\{X _{t} \}, \beta) \right) \cap \left(L
   _{\beta}  M \times (0, a) \times \{0\} \right), $$
       there is a compact open subset $\mathcal{C} _{y} \ni y $ of $S$ contained in  
     $M \times (0,b) \times
  [0,1]$.
%
Respectively,
if for each $$y \in S \cap \left(L _{\beta}  M \times (0, a) \times \partial [0,1]
   \right), $$
there is a compact open subset $\mathcal{C} _{y} \ni y $ of $S$ contained in  
$M \times (0,b) \times
[0,1]$.
 \end {definition}
\begin{lemma} \label{lemma:partiallyadmissible} Suppose that $\{X _{t} \}$ is
partially admissible,  then for every $a$  
there is a $b>a$ so that $\{\widetilde{X}  ^{b}_t   \} =  \{X _{t}
 \} \cdot \{X ^{b} _ t
 \} $
 is partially $a,b$-admissible, where $\{X _{t}
 \} \cdot \{X ^{b} _ t
 \} $ is the (reparametrized to have $t$ domain $[0,1]$) concatenation of the homotopies $\{X _{t} \}, \{X
 ^{b} _{t}  \}$, and where  $\{X  ^{b} _ t  \}$ is the structure
 homotopy from 
 $X  ^{b  }  $ to $X _{0} $.
\end{lemma}
\begin{proof}
More explicitly
\begin{equation}  
\begin{aligned} 
    \{\widetilde{X}  ^{b}_t   \} & =  \{X ^{b} _{2t}  \} \text{ for } t \in [0,1/2] \\
    \{\widetilde{X}  ^{b}_t   \} & =  \{X _{2(t - 1/2)} \} \text{ for } t \in [1/2,1].
\end{aligned}
\end{equation}   
 Let $y \in S (\{X ^{a} _{t}  \})  \cap \left(L
   _{\beta}  M \times (0, a) \times \{0\} \right) $.
Let $\mathcal{C}' _{y} $ be a non-branching open compact subset of
 $S (\{X ^{a}_{t}  \})$ containing $y$.
Let $$y' \in K _{y} =  \mathcal{C}' _{y}  \cap \left(L _{\beta}M  \times (0, \infty) \times
\{1\} \right),$$ with $K _{y} $ by assumptions connected and since it is an open and
closed subset of $$S (\{X ^{a} _{t}  \})  \cap \left(L _{\beta}M  \times (0, \infty) \times
\{1\} \right),$$ $\pi (K _{y} )$ coincides with one of the connected components of $S (X
_0) $, 
for $$\pi: L _{\beta}M \times (0, \infty)
 \times [0,1] \to L _{\beta}M \times (0, \infty)$$ the projection.

Let $\mathcal{C} _{y'} $ be a compact open subset of
 $S (\{X_{t}  \})$ containing $\pi (y') \times \{0\} \in L _{\beta}M  \times (0, \infty)
 \times [0,1]$, $\mathcal{C} _{y'} $
 exists by partial admissibility
 assumption on $\{X _{t} \}$.
Let $$ M _{y} = \mathcal{C} _{y'} \cap \left( L _{\beta}M \times (0, \infty) \times
\{0\} \right)
 $$ then $M_{y} $ must contain $\pi(K _{y}) \times \{0\} $ as these sets are 
   open and closed in $S (\{X _{t} \}) \cap \left(L _{\beta}M  \times (0, \infty) \times
\{0\} \right)$ and $\pi(K _{y}) \times \{0\} $ is
   connected.
And $M _{y} 
 - \pi (K _{y}) \times \{0\}   $ is a finite union (possibly empty) of compact
 open connected components
 $\{W ^{j} 
 _{y} \}$, by the assumption that connected components of $S (X)$ are open, and
 since $M _{y} $ is compact.

Let $$S _{a} =  \bigcup _{y} \mathcal{C} _{y'}$$ for $y, y'$ as above. 
Then since there are only finitely many such $y$ $S _{a} $ is compact and so
is contained in  $L _{\beta} M \times (0, b') \times
 [0,1]$, for some $b' > a$ and sufficiently large as in the last axiom for a
 perturbation system.

For each $W ^{j} _{y}  $ as above let $\mathcal{C} ^{j} _{y}   $ be
a non-branching open compact subset of $S (\{X ^{b'} _{t}  \} )$ intersecting $\pi (W ^{j}
 _{y}) \times \{1\}  $, and hence so that $$\mathcal{C} ^{j} _{y} \cap \left (L _{\beta}
 M \times (0, \infty) \times
 \{1\} \right) = \pi (W ^{j}
 _{y}) \times \{1\}.  $$ This equality again follows by $\mathcal{C} ^{j} _{y} \cap \left(L _{\beta}
 M \times (0, \infty) \times
 \{1\} \right)$ and
 $\pi (W ^{j}
 _{y}) \times \{1\}$ being open closed and connected subsets of 
$S (\{X ^{b'}_{t}  \}) \cap \left (L _{\beta}
 M \times (0, \infty) \times
 \{1\} \right)$.

Let
\begin{equation*}
   T _{y} =   (\bigcup _{j}  \mathcal{C} ^{j} _{y}) \cup  \mathcal{C}' _{y} \cup
   S _{a},
\end{equation*}
where this union is taken  in 
\begin{equation} \label{eqUnion}
   \mathcal{U} =  \mathcal{U} _{-}   \sqcup \mathcal{U} _{+}/\sim,
\end{equation}
where $\mathcal{U} _{\pm} $ are two names for $L _{\beta}M \times (0, \infty)
   \times [0,1]$ 
and the equivalence relation the identification map of $M \times (0, \infty)
 \times \{1\}$ in the first component with  $M \times (0, \infty)
 \times \{0\}$ in the second component. And here $(\bigcup _{j}  \mathcal{C} ^{j}
 _{y}) \cup  \mathcal{C}' _{y}$ is understood as being a subset of
 $\mathcal{U}_{-} $
 and $S _{a}$ of $\mathcal{U} _{+} $. 
 
 
Let $$Q=[0,1] _{-}  \sqcup [0,1] _{+} /
\sim,$$ be the quotient space, where $[0,1] _{\pm} $ are two names for $[0,1]$,
and $\sim$ 
identifying $1 \in [0,1] _{-}  $  with $0 \in [0,1] _{+} $.
Let \begin{equation*}
\phi: Q \to [0,1],
\end{equation*} 
be the ``linear'' (linear, if one naturally identifies  $Q$ with $[0,2]$)
homeomorphism with $\phi (0)=0$, for
 $0 \in [0,1] _{-} $, $\phi (1) = 1/2$ for $1 \in [0,1] _{-} $,  $\phi (0)
 = 1/2$, for $0 \in [0,1] _+$, and $\phi (1) =1$, for $1 \in [0,1] _{+} $.
Then by the above discussion $T _{y} $ is a compact subset of \eqref{eqUnion}, 
and $C _{y} = \widetilde{\phi} (T _{y})   $
is a compact and open subset of
$ S (\{ \widetilde{X} ^{b} _{t}
\})$,  containing $y$, where $$\widetilde{\phi}: \mathcal{U} \to L _{\beta}M  \times (0, \infty)
 \times [0,1],$$ is induced by   $\phi$.
Again since there are only finitely many such $y$ \begin{equation*}
   \bigsqcup _{y} C _{y}  
\end{equation*}
is contained in $L _{\beta}M  \times (0, b) \times
 [0,1]$, for some $b$ sufficiently large.
 So our assertion follows.  %
\end{proof}
The analogue of Lemma \ref{lemma:partiallyadmissible} in the admissible case is the following:
\begin{lemma} Suppose that $X _{0}, X _{1}  $ and $\{X _{t} \}$ are admissible,
  then for every $a$  
there is a $b>a$ so that $\{\widetilde{X}  ^{b}_t   \} = \{ {X} ^{b}
_{1,t}   \} ^{-1} \cdot \{X _{t}
 \} \cdot \{X ^{b} _ {0,t}
 \} $
 is $a,b$-admissible,  where  $\{X  ^{b} _ {i,t}  \}$ are the structure
 homotopies from 
 $X _{i}   ^{b}  $ to $X _{i} $.
\end{lemma}
The proof of this is completely analogous to the proof of Lemma
\ref{lemma:partiallyadmissible}.

\subsection {Proof of Theorem \ref{thm:welldefined}}

Suppose that $X _{0} $ is admissible with $i (X _{0}, \beta ) \neq 0$, $\{X _{t} \}$ is partially admissible and
$X _{1} $ has no periodic orbits.
Let $a$  be given  and $b$ determined so that   $\{\widetilde{X}  ^{b}_t   \} $
is a partially $(a,b) $-admissible homotopy.
Set $\textbf{m} = \textbf{m} (\{ \widetilde{X} ^{b} _{t} \}, b)$. Take a prime $k = k (a)> \textbf{m}$ and define $\textbf{M}$ as above with
respect to $k$. 
 As $M$ is compact the flow $\widetilde{X} ^{b} _{t} $ is complete
   for every $t$, and consequently as previously observed the flow of $\textbf{X} _{t} $ is complete
   for every $t$.

Let $\textbf{F} _{t,p} $ denote the time $p$ flow map of $\textbf{X} _{t} $.
Define
$$\textbf{F}: \textbf{M}  \times (0, b] \times [0,1] \to
 \textbf{M} \times \textbf{M}, $$ by
\begin{equation*}
\textbf{F} (\textbf{x}, p, t) = (\textbf{F} _{t,p} (\textbf{x}), \textbf{x} ).
\end{equation*}
Set
\begin{equation} \label{eqWTS}
   \widetilde{\textbf{S}}   =  \textbf{F} ^{-1} (\Delta),
%
   %
\end{equation}
for $\Delta$ the diagonal.
And set 
\begin{equation*}
   \textbf{S} = emb \circ Ful _{k}  (S (\{ \widetilde{X}_{t} ^{b}   \}, \beta )) \subset
   \textbf{M} \times (0, \infty) \times [0,1],
\end{equation*}
for 
\begin{equation} \label{eqEmb}
emb: Ful _{k}  (S (\{ \widetilde{X}_{t} ^{b}   \}, \beta )) \to \textbf{M} \times (0,\infty)
\times [0,1],
\end{equation}
 the map $(\textbf{o},p, t) \mapsto (\textbf{o} (0),p,t) $.
Let
$$\textbf{S} ^{i}    = (\textbf{S} \cap \textbf{M} \times \mathbb{R}_+ \times
\{i\}).$$
By assumptions $\textbf{S} ^{1} $ is empty. 

 Let $\{X _{t} \}$ be partially $a,b$-admissible and 
let 
\begin{equation} \label{eqS}
S _{a, \beta, 0} =  \bigcup _{y \in S (\{X _{t} \}, \beta)   \cap \left( M
\times (0, a) \times
 \{0\} \right)} \mathcal{C} _{y}.
\end{equation}
 Here $\mathcal{C} _{y} $ is as in the Definition \ref{def:admissible} with
 respect to $a,b$. Then $S _{a, \beta, 0}$ is an open compact subset of  $S (\{X
 _{t} \}, \beta) \cap \left(L _{\beta} M
\times (0, b) \times
 [0,1] \right)$.
Define $${\textbf{S}}  _{a,0 } ={\textbf{S}}  _{a, \beta, 0} =  emb \circ Ful
_{k}(S _{a, \beta, 0}).
$$ 

A given orientation on $M$ orients $\textbf{M}$, $\textbf{M}
 \times \textbf{M}$, and the diagonal $\Delta \subset \textbf{M} \times \textbf{M}$. Let
 $\{\nabla _{r}\}  $ be a sequence of forms 
$C ^{\infty} $ dual to
the diagonal $\Delta$, with support of $\nabla _{r} $ converging to the
diagonal as $r \mapsto \infty$, uniformly on compact sets.
   $\{\nabla _{r} \}$ are characterized by the condition that $C \cdot \Delta= \int _{C} f ^{*} \nabla _{r}    $, where $f: C \to \textbf{M} \times \textbf{M}$ is a
chain whose boundary is disjoint from $\Delta$, and from the support of $\nabla
   _{r} $ and where $C \cdot \Delta$ is
the intersection number.

\begin{lemma} \label{lemma:breakup}
   We may choose an
   $r$ so that  $\textbf{F} ^ {\,*}  \nabla _{r}  $  breaks up as the sum:
\begin{equation*}
\gamma ^{r}   + \sum _{y} \alpha ^{r}   _{y},
\end{equation*}
for $y$ as in \eqref{eqS}, where each $\alpha ^{r}    _{y} $ has compact support
   in open sets $U _{y} $, $$\textbf{M} \times (0,\frac{b}{k}) \times [0,1] \supset U
 _{y} \supset
   emb \circ Ful _{k} (\mathcal{C} _{y}), $$ and where $\gamma ^{r} $ has support which does
not intersect any of the $U _{y}  $.
\end{lemma}
  \begin{proof}  
Note that ${\textbf{S}} _{a,0 } $ is an open and compact subset of
 $\widetilde{\textbf{S}} \cap \left(\textbf{M} \times (0,
\frac{b}{k})  \times [0,1] \right)$  by
     construction, and by the fact that the maps \eqref{eqEmb} are open. (The
     latter, by the
     definition of the topology on these spaces as discussed following
     \eqref{eqSfirst}).
     Likewise $$ \widetilde{ \textbf{S}}  - {\textbf{S}} _{a,0 } $$ is compact as $ \widetilde{\textbf{S}}  $ is compact,  and ${\textbf{S}} _{a,0
     }$ is open.

Let $sup$ denote the support
of  $\textbf{F} ^ *  \nabla _{r}  $. Then for $r$ sufficiently large $sup$ is contained
     in an $\epsilon$-neighborhood of $ \widetilde{\textbf{S}}$, for $\epsilon$
     arbitrarily small (just by compactness).
Since ${\textbf{S}} _{a,0 } $ and $\widetilde{ \textbf{S}}  -
     {\textbf{S}} _{a,0 } $ are both compact and disjoint they have disjoint
   metric $\epsilon$-neighborhoods for $\epsilon$ sufficiently small.
The lemma then clearly follows.
%
 \end{proof} 
As  $
\gamma ^{r}   + \sum _{y} \alpha ^{r}  _{y}
$ is closed, $\omega ^{r}= \sum _{y} \alpha ^{r}  _{y}$ is closed and  has
compact support. 
   Let $\omega ^{r} _0, \omega ^{r}  _{1}  $ be the restrictions of $\omega ^{r} $ to $\textbf{M} \times \mathbb{R} _{+} \times 
\{0\}$, respectively  $\textbf{M} \times \mathbb{R} _{+} \times 
\{1\}$, with $\omega ^{r}  _{1} $ by assumption identically vanishing, for $r$
sufficiently large.

Let $S
(\textbf{X} ^{b}, b,\beta)/S ^{1} $, be as in \eqref{eqFulS}, whose elements by slight abuse of notation we
denote by 
$\textbf{o} = (\textbf{o}, p)$.
Let $[(\omega ^{r} _{0})  ^{*} ]  $ denote the Poincare dual class of $\omega
^{r}  _{0} $
and given  $\textbf{o} $ as above we denote by
$[\textbf{o}]$ the class of the
1-cycle in $\textbf{M} \times \mathbb{R} _{+} $ represented by the (strictly
speaking $S ^{1} $-equivalence class of) map $t 
\mapsto (\textbf{o}(t), p)$, $t \in [0,p]/0 \sim p$. If $r$ is taken to be
large, then by construction and since all elements of $S
(\textbf{X} ^{b}, b, \beta)/S ^{1} $,  are isolated,
the support of $\omega ^{r}  _{0}$  breaks up as the disjoint union
of 
sets contained in $\epsilon _{r} $-neighborhoods of the images of the orbits $\textbf{o} \in  S
(\textbf{X} ^{b}, b, \beta)/S ^{1}  $, for $\epsilon _{r} \mapsto 0 $ as $r
\mapsto \infty$.
We shall say for short that the support is \emph{localized} at these images.
So  we may
write $ \omega ^{r}  _0= \sum _{l}
\omega   _{0,l}$, with $\omega   _{0,l}  $ having support localized at
the image of  $ 
\textbf{o} ^{0} _{l} $, where $\{\textbf{o} ^{0} _{l}  \}$ is the enumeration of $S
(\textbf{X} ^{b}, b, \beta)/S ^{1} $.

By
(proof of) \cite [Theorem 1]{citeFullerIndex} (for $r$ sufficiently large) the
Poincare dual class $[\omega  _{0,l}
    ^*]  $ is given by the class
$$i (\textbf{o} ^{0}_l ) \frac{1}{\mult (\textbf{o}  ^{0} _{l}  )}
   [\textbf{o} ^{0} _{l}]
   ,$$ and by Fuller's correspondence $i (\textbf{o} ^{0}_l ) = i ({o} ^{0}_l
   ) $, and $\mult (\textbf{o}  ^{0} _{l}  ) = \mult (o ^{0} _{l}  )$, where $o
   ^{0} _{l}  $ is the corresponding element of $S(X ^{b}_0, \beta, b) $.
So  
\begin{equation*}
  [(\omega _{0} ^{r})  ^{*}] = c _{a} + \sum _{\textbf{o} ^{0} \in S
(\textbf{X} ^{a}, a,\beta)/S ^{1}  } i (o ^{0} )\frac{1}{\mult (o ^{0} )}
   [\textbf{o} ^{0}],
\end{equation*} 
where
$$c _{a} = \sum _{\textbf{o} \in Q _{a} } i (o)\frac{1}{\mult (o  )}
[\textbf{o}],
     $$ for $$Q _{a}  \subset S
(\textbf{X} ^{b}, b,\beta)/S ^{1} - S
(\textbf{X} ^{a}, a,\beta)/S ^{1}    $$ (possibly empty).
Thus
\begin{equation*}
   c _{a} + \sum _{\textbf{o} ^{0} \in S
(\textbf{X} ^{a}, a,\beta)/S ^{1} } i (o ^{0} )\frac{1}{\mult (o ^{0}
      )} [\textbf{o} ^{0}] = 0,
\end{equation*}
as for $r$ sufficiently large $\omega ^{r} _0, \omega ^{r}  _{1}  $  are cohomologous with compact support
in $\textbf{M} \times (0, \frac{b}{k})  \times [0,1]$ and as $\omega ^{r}
_{1}$ is identically vanishing. 
Applying  $\mu$ we get 
\begin{equation} \label{eq:contradiction}
   l (\sum _{\textbf{o}  \in Q _{a} } i (o)\frac{1}{\mult (o  )}  + \sum _{\textbf{o} ^{0} \in S
(\textbf{X} ^{a}, a,\beta)/S ^{1}} i ({o} ^{0}
   )\frac{1}{\mult ({o} ^{0} )})  = 0 \mod k,
\end{equation} 
where $l$ is the least common denominator for all the fractions,
and this holds for all $a, k = k (a)$ (going higher in the perturbation system and
adjusting the least common denominator). 
\subsection {Case I, $X _{0} $ is finite type} Let $E=E (\{X ^{a} \})$ be the
corresponding cutoff value in the definition of finite type, and take any 
$a > E$.
Then $Q _{a} = \emptyset $ and
\begin{equation*}
\sum _{\textbf{o}  \in Q _{a} } i (o)\frac{1}{\mult (o  )}  + \sum _{\textbf{o} ^{0} \in S
(\textbf{X} ^{a}, a,\beta)/S ^{1} } i ({o} ^{0}
   )\frac{1}{\mult ({o} ^{0} )} = i (X _{0}, \beta ) \neq 0.
\end {equation*}
Clearly this gives a contradiction to \eqref{eq:contradiction}.

\subsection {Case II, $X _{0} $ is infinite type} We may assume that $i (X _{0},
\beta) =
\infty$, and take $a > E$, where $E=E (\{X ^{a} \})$ is the corresponding cutoff
value in the definition of infinite type.
Then 
\begin{equation*}
\sum _{\textbf{o}  \in Q _{a} } i
(o)\frac{1}{\mult (o  )} \geq 0,
\end {equation*}
as $a > E (\{X _{0}  ^{e} \} )$.
While 
\begin {equation*}
\lim _{a \mapsto \infty}  \sum _{\textbf{o} ^{0} \in S
(\textbf{X} ^{a}, a,\beta)/S ^{1} } i ({o} ^{0}
   )\frac{1}{\mult ({o} ^{0} )} = \infty,
\end{equation*}
by $i (X _{0}, \beta ) = \infty$.
This also contradicts \eqref{eq:contradiction}.

\qed

\subsection {Proof of Theorem \ref{thm:welldefinedadmissible}}
The proof is very similar to the proof of Theorem \ref{thm:welldefined}, and all
the same notation is used.
Suppose that $X _{0}, X _{1}  $  and $\{X _{t} \}$ are admissible. Let $a$ be given and $b$ determined so that   $\{\widetilde{X}  ^{b}_t   \} $ is a  $(a,b)
 $-admissible homotopy.  
 Let 
$\textbf{F}$, 
$\widetilde{ \textbf{S}}$,  $\textbf{S}$ and
$\textbf{S} ^{i}$,    
 be as before. 
  
 Let $$S _{a,\beta} = S _{a, \beta} (F) = \bigcup _{y \in S _{\beta} (F)   \cap \left( M
\times (0, a) \times
 \partial [0,1] \right)} \mathcal{C} _{y},
 $$ 
where $\mathcal{C} _{y} $ are as in the Definition \ref{def:admissible}.
Then $S _{a, \beta, 0}$ is an open compact subset of  $S (\{X
 _{t} \}, \beta) \cap \left (L _{\beta} M
\times (0, b) \times
 [0,1] \right)$.
Define $${\textbf{S}}  _{a} ={\textbf{S}}  _{a, \beta} = emb \circ Ful _{k} (S
_{a, \beta}).
$$  
 Let
 $\{\nabla _{r}\}  $ be a sequence of forms 
$C ^{\infty} $ dual to
the diagonal $\Delta \subset \textbf{M} \times \textbf{M}$, as before.
\begin{lemma} 
   We may choose an
   $r$ so that  $\textbf{F} ^ {\,*}  \nabla _{r}  $  breaks up as the sum:
\begin{equation*}
\gamma ^{r}   + \sum _{y} \alpha ^{r}   _{y},
\end{equation*}
for $y$ as in \eqref{eqS}, where each $\alpha ^{r}    _{y} $ has compact support
   in open sets $U _{y} $, $$\textbf{M} \times (0,\frac{b}{k}) \times [0,1] \supset U
 _{y} \supset
   emb \circ Ful _{k} (\mathcal{C} _{y}), $$ and where $\gamma ^{r} $ has support which does
not intersect any of the $U _{y}  $.
\end{lemma}
  \begin{proof}  
  Analogous to the proof of Lemma \ref{lemma:breakup}.
 \end{proof} 
As  $
\gamma ^{r}   + \sum _{y} \alpha ^{r}  _{y}
$ is closed, $\omega ^{r}= \sum _{y} \alpha ^{r}  _{y}$ is closed and  has
compact support. 
   Let $\omega ^{r} _0, \omega ^{r}  _{1}  $ be the restrictions of $\omega ^{r} $ to $\textbf{M} \times \mathbb{R} _{+} \times 
\{0\}$, respectively  $\textbf{M} \times \mathbb{R} _{+} \times 
\{1\}$.
   Let $\omega ^{r} _0, \omega ^{r}  _{1}  $ be the restrictions of $\omega ^{r} $ to $\textbf{M} \times \mathbb{R} _{+} \times 
\{0\}$, respectively  $\textbf{M} \times \mathbb{R} _{+} \times 
\{1\}$.
We may
write $ \omega ^{r}  _i= \sum _{l}
\omega   _{i,l}$, with $\omega   _{i,l}  $ having support localized at
the image of  $ 
\textbf{o} ^{i} _{l} $, where $\{\textbf{o} ^{i} _{l}  \}$ is the enumeration of $S
(\textbf{X} _{i}  ^{b}, b, \beta)/S ^{1} $.

By
(proof of) \cite [Theorem 1]{citeFullerIndex} (for $r$ sufficiently large) the
Poincare dual class $[\omega  _{i,l}
    ^*]  $ is given by the class
$$i (\textbf{o} ^{i}_l ) \frac{1}{\mult (\textbf{o}  ^{i} _{l}  )}
   [\textbf{o} ^{i} _{l}]
   ,$$ and by Fuller's correspondence $i (\textbf{o} ^{i}_l ) = i ({o} ^{i}_l
   ) $, and $\mult (\textbf{o}  ^{i} _{l}  ) = \mult (o ^{i} _{l}  )$, where $o
   ^{i} _{l}  $  is the corresponding element of $S(X ^{b}_i, \beta, b) $.

So  
\begin{equation*}
  [(\omega _{0} ^{r})  ^{*}] = c _{a} + \sum _{\textbf{o} ^{0} \in S
(\textbf{X} _{0}  ^{a}, a,\beta)/S ^{1}  } i (o ^{0} )\frac{1}{\mult (o ^{0} )}
   [\textbf{o} ^{0}],
\end{equation*} 
where
$$c _{a} = \sum _{\textbf{o} \in Q _{a} } i (o)\frac{1}{\mult (o  )}
[\textbf{o}],
     $$ for $$Q _{a}  \subset S
(\textbf{X} _{0}  ^{b}, b,\beta)/S ^{1} - S
(\textbf{X} _{0}  ^{a}, a,\beta)/S ^{1}    $$ (possibly empty).

Likewise 
\begin{equation*}
  [(\omega _{1} ^{r})  ^{*}] = c' _{a} + \sum _{\textbf{o} ^{1} \in S
(\textbf{X} _{1}  ^{a}, a,\beta)/S ^{1}  } i (o ^{1} )\frac{1}{\mult (o ^{1} )}
   [\textbf{o} ^{1}],
\end{equation*} 
where
$$c' _{a} = \sum _{\textbf{o} \in Q' _{a} } i (o)\frac{1}{\mult (o  )}
[\textbf{o}],
     $$ for $$Q' _{a}  \subset S
(\textbf{X} _{1}  ^{b}, b,\beta)/S ^{1} - S
(\textbf{X} _{1}  ^{a}, a,\beta)/S ^{1}    $$ (possibly empty).
We have: 
\begin{equation*}
   c _{a} + \sum _{\textbf{o} ^{0} \in S
(\textbf{X} _{0}  ^{a}, a,\beta)/S ^{1} } i (o ^{0} )\frac{1}{\mult (o ^{0}
      )} [\textbf{o} ^{0}
] = c' _{a} + \sum _{\textbf{o} ^{1} \in S
(\textbf{X} _{1}  ^{a}, a,\beta)/S ^{1} } i (o ^{1} )\frac{1}{\mult (o ^{1}
      )} [\textbf{o} ^{1}
],
\end{equation*}
as $\omega ^{r} _0, \omega ^{r}  _{1}  $  are cohomologous with compact support
in $\textbf{M} \times (0, \frac{b}{k})  \times [0,1]$.
Applying  $\mu$ we get \begin{multline} \label{eq:equalitymodk}
   l (\sum _{\textbf{o}  \in Q _{a} } i (o)\frac{1}{\mult (o  )} + \sum _{\textbf{o} ^{0} \in S
(\textbf{X} _{0}  ^{a}, a,\beta)/S ^{1} }  i ({o} ^{0}
   )\frac{1}{\mult ({o} ^{0} )})  \\= l(\sum _{\textbf{o}  \in Q _{a} '} i (o)\frac{1}{\mult (o  )} + \sum _{\textbf{o} ^{1} \in S
(\textbf{X} _{1}  ^{a}, a,\beta)/S ^{1} }  i ({o} ^{1}
   )\frac{1}{\mult ({o} ^{1} )}) \mod k
\end{multline}
where $l$ is the least common denominator for all the fractions, and this holds
for all $a,k$ (changing $l$ appropriately and going higher in the perturbation system).

Suppose by contradiction that $i (X_0, \beta)  \neq i
(X_1, \beta) $.
\subsection* {Case I, $X _{i} $ are finite type} Let $E _{i} = E (\{X ^{a} _{i}
\}) $ be the corresponding cutoff values, and take $a> \max (E _{0}, E _{1}  )$.
Then $Q _{a} = Q' _{a} = \emptyset  $ and
\begin{equation*}
\sum _{\textbf{o}  \in Q _{a} } i (o)\frac{1}{\mult (o  )} + \sum _{\textbf{o} ^{0} \in S
(\textbf{X} _{0}  ^{a}, a,\beta)/S ^{1} }  i ({o} ^{0}
   )\frac{1}{\mult ({o} ^{0} )} = i (X _{0}, \beta ),
\end{equation*}
and
\begin{equation*}
\sum _{\textbf{o}  \in Q _{a} '} i (o)\frac{1}{\mult (o  )} + \sum _{\textbf{o} ^{1} \in S
(\textbf{X} _{1}  ^{a}, a,\beta)/S ^{1} }  i ({o} ^{1}
   )\frac{1}{\mult ({o} ^{1} )} = i (X _{1}, \beta ).
\end{equation*}
This gives a contradiction to \eqref{eq:equalitymodk}.
\subsection* {Case II, $X _{i} $ are infinite type}
Let $E _{i} = E (\{X ^{a} _{i}
\}) $ be the corresponding cutoff values, and take $a> \max (E _{0}, E _{1}  )$.
Suppose in addition (WLOG) that $i
(X_0) = \infty$, $i
(X_1) = - \infty$. 
Then $$\sum _{\textbf{o}  \in Q _{a} } i
(o)\frac{1}{\mult (o  )} \geq 0, $$
and 
\begin{equation*}
 \sum _{\textbf{o}  \in Q _{a} '} i
(o)\frac{1}{\mult (o  )} \leq 0,
\end{equation*}
while 
\begin{equation*}
\lim _{a \mapsto \infty} \sum _{\textbf{o} ^{0} \in S
(\textbf{X} _{0}  ^{a}, a,\beta)/S ^{1} } i ({o} ^{0}
   )\frac{1}{\mult ({o} ^{0} )}) = \infty,
\end{equation*}
and 
\begin{equation*}
\lim _{a \mapsto \infty}  \sum _{\textbf{o} ^{1} \in S
(\textbf{X} _{1}  ^{a}, a,\beta)/S ^{1} } i ({o} ^{1}
   )\frac{1}{\mult ({o} ^{1} )} = - \infty.
\end{equation*}

Clearly this also gives contradiction.
\subsection* {Case III, $X _{0} $ is infinite type and $X _{1} $ is finite type} 
Let $E _{i} = E (\{X ^{a} _{i}
\}) $ be the corresponding cutoff values, and take $a> \max (E _{0}, E _{1}  )$.
Suppose in addition that $i
(X_0) = \infty$.  
Then $$\sum _{\textbf{o}  \in Q _{a} } i
(o)\frac{1}{\mult (o  )} \geq 0, $$

while 

\begin{equation*}
\lim _{a \mapsto \infty} \sum _{\textbf{o} ^{0} \in S
(\textbf{X} _{0}  ^{a}, a,\beta)/S ^{1} } i ({o} ^{0}
   )\frac{1}{\mult ({o} ^{0} )} = \infty,
\end{equation*}

and $Q ' _{a} = \emptyset $ so

$$\sum _{\textbf{o}  \in Q _{a} '} i (o)\frac{1}{\mult (o  )} + \sum _{\textbf{o} ^{1} \in S
(\textbf{X} _{1}  ^{a}, a,\beta)/S ^{1}} i ({o} ^{1}
   )\frac{1}{\mult ({o} ^{1} )} = i (X _{1}, \beta ).$$
Again this is a contradiction.

\qed
\section {Proof of Lemma \ref{lemmaExample}}
\begin{lemma} Let $X$ be a smooth a non-singular vector field on a manifold
   $M$,
and ${F} _{p} $ denote the time $p$ flow map of $X$.
   Set \begin{equation*} 
   {{S}} (X)   =  \{({x},p) \in {M} \times (0, \infty)  \, | \, {F}  _{p}  ({x},p) = {x}\}.
\end{equation*} 
Suppose that $S$ consists of compact isolated components $\{S _{i} \}$, meaning that
for every $S _{i _{0} } \in \{S _{i} \} $ there exists an $\kappa > 0$ so that
   there is a neighborhood $U _{i} $ of $S _{i _0} $ in ${M} \times (0,
   \infty)$, whose closure $\overline U _{i}$ is compact and so that $\overline U _{i} \cap S (X) = S _{i}   $
   . Then for every
   $\epsilon > 0$ there exists
   an $\delta$ so that whenever  $X
   _{1} $ is a smooth vector field $C ^{0} $ $\delta$ close to $X$, and $p \in S (X
   _{1} )$ is contained in $\overline {U} _{i} $, then $p $ is in the
   $\epsilon$-neighborhood of $S _{i}  $.
\end{lemma}
\begin{proof} 
    Let
   $${F} (X): {M}  \times (0, \infty)  \to
 {M} \times {M}, $$ be the map
\begin{equation*}
{F} (X) ({x}, p) = ({F} _{p} ({x}), {x} ),
\end{equation*} 
then $S (X)$ is the preimage of the diagonal $\Delta$ by $F (X)$.
   Let
   $\epsilon$ be given, so that the $\epsilon$-neighborhood $V _{i} $ of $S _{i}
   $ in $M \times (0, \infty)$ is
   contained in $U _{i} $. 
   Suppose otherwise by contradiction, then there exists a sequence $\{X _{i} \} $ of
   vector fields $C ^{0} $ converging to $X$, and a sequence $\{p _{i} \}
   \subset U _{i} - V _{i}  $, $p _{i} \in S (X _{i} )  $. We may then find a
   convergent subsequence $\{p _{i _{k} }\} \mapsto p \in \overline {U} _{i} - V
   _{i}   $. But $\{F (X _{i _{k} } )\}$ is uniformly on $\overline {U} _{i}$
   convergent to $F (X)$, so that $F (X)
   (p) \in \Delta$ and so $p \in S (X)$. But this is a
   contradiction to the hypothesis that $\overline U _{i} \cap S (X) = S _{i}   $.
\end {proof}
Lemma \ref{lemmaExample} then readily follows. Let us leave out the details.
\qed
\section  {Proof of Theorem \ref{prop:AdmissibleReeb}}
  
 Suppose that $\lambda _{t} =f _{t}
   \lambda$, $f_t>0$, and let $X _{t} $ be the Reeb vector field for $\lambda _{t} $.
  Let $p: [0, \infty) \to S $ be a locally Lipschitz path, so that $p_3 =\pi _{3}
  \circ p $ has finite length $L$. This means that the metric derivative function
  $|\frac{d}{d \tau} p _{3}|$:
 \begin{equation*}
 |\frac{d}{d \tau} p _{3}   (\tau _{0} )| = \limsup _{s \mapsto \tau _{0} }
 |\rho _{3} (\tau _{0} ) - \rho _{3}
 (s) |/ |\tau _{0}  -s|,
 \end{equation*} 
satisfies
\begin{equation*}
\int _{0} ^{\infty}  |\frac{d}{d\tau  } p _{3}|  \, d\tau = L < \infty.
\end{equation*}
Note that $|\frac{d}{d\tau} p _{3}| $ is always locally Riemann integrable by the
Lipschitz condition.  From now on if we write expression of the form
$|\frac{d}{dx} f (x _{0} )|$ we shall mean the metric derivative evaluated at $x
_{0} $.


Consequently, for any $D>0$, we may reparametrize (with same domain)
  $\rho =p|_{[0, D]}$ so that the reparametrized path (notationally unchanged) satisfies:
 \begin{equation*}
 |\frac{d}{d \tau } \rho _{3}  (\tau _{0} )| \leq \frac{L}{ D}, \text{ for almost all
 $\tau _{0} $}, \quad \rho _{3} =
 \pi _3 \circ \rho.
 \end{equation*} 


Set
  $$K = \max _{t \in [0,1], M} |\frac{df
   _{t}}{dt}| \cdot
     \max _{t,M} f _{t}.  $$ 
Let $\rho _{i} = \pi _{i} \circ \rho $, for $\pi _{i} $ the projections of $LM
\times (0, \infty) \times [0,1]$ onto the $i$'th factor.
     Clearly our theorem follows by the following lemma.


   \begin{lemma} \label{lemma:delta}
   The metric length of the path $\rho _{2} = \pi _{2} \circ \rho  $ is bounded
   from above by $\exp (L \cdot K)$ for any $D$.
 \end{lemma} 
\begin{proof}

After
reparametrization of $\rho$ to have domain in $[0,1]$ and keeping the same notation for the path,
we have
\begin{equation} \label{eq:upperbound}
|\frac{d}{d\tau} \rho _{3} (\tau _{0} )| \leq L, \text{ for almost all $\tau _{0} $ }.
\end{equation}

 For every $\tau$ we
 have a loop $\gamma _{\tau}: [0, 1] \to M$ given by $$\gamma _{\tau}
 (t) = F ^{\rho _{3} (\tau)  }  _{ \rho _{2} (\tau) \cdot t} (\rho _1
 (\tau)),$$ where the flow maps $F$ are defined as before. In other words
 $\gamma _{\tau} $ is the closed orbit of $\rho _{2} \cdot R ^{\lambda _{\rho
 _{3} (\tau) } }  $ - a $R ^{\lambda _{\rho
 _{3} (\tau) } }$-Reeb orbit.
 So we get a locally Lipschitz path
 $\widetilde{\rho}$ in $L  M \times [0,1]$, $\widetilde{\rho} (\tau) =
 (\gamma _{\tau}, \rho_3(\tau) )$, where $L M$ is the
 free loop space of $M$ with uniform metric.



We have the smooth functional:
$$\Lambda: L M \times [0,1]  \to
\mathbb{R}, \quad \Lambda (\gamma, t) =  \langle \lambda _{t},  \gamma \rangle, $$
where $ \langle ,  \rangle $ denotes the integration pairing, and clearly $\Lambda
(\widetilde{\rho} (\tau) ) = \rho _{2} (\tau)  
$.

We also have the
restricted functionals $\lambda_{t}:  L M   \to
\mathbb{R}, \quad \lambda _{t}  (\gamma) = \langle \lambda _{t},  \gamma \rangle, $ called
the $\lambda _{t} $ functionals.
Let $\xi (\tau) = {\pi} _{1,*} \frac{d}{d \tau} \widetilde{\rho}
(\tau)  $, for ${ \pi}
_{i}$ the projections of $ L M \times [0,1]  $ onto the $i$'th factor. 
The differential of $\Lambda$ at $(o, t _{0} ) \in LM \times [0,1]$ is
\begin{equation*}
D  \Lambda (o, t _{0} ) (\xi, \frac{\partial}{\partial \tau}) = \left . D \lambda _{t} (o) (\xi)
+ \frac{d}{d t} 
\right | _{t _{0} } \lambda _{t} (o),
\end{equation*}
where $\xi \in T _{o} LM $, 
and if $o$ is a $R ^{\lambda _{t _{0} } } $- Reeb orbit the first term vanishes.
Consequently, 
\begin{equation*}
|\frac{d}{d \tau} \Lambda \circ \widetilde{\rho} (\tau _{0} )| = |\frac{d}{d \tau}
\lambda _{ \pi _{2} (\tau) } (\pi _1 \circ \widetilde{\rho} (\tau _{0} ) )|.
\end{equation*}

On the other hand 
\begin{equation*}
|\frac{d}{d \tau}
\lambda _{ \pi _{2} (\tau) } (\pi _1 \circ \widetilde{\rho} (\tau _{0} ) )| \leq
|\frac{d}{d t} 
  \lambda _{t} ( \widetilde{\rho} (\tau _{0} ) )
({\rho _{3} (\tau _{0}) })| \cdot
|\frac{d}{d\tau} \rho _{3} (\tau _{0} )|.
\end{equation*}
 We have by direct calculation for any $\tau _{0} \in [0,1] $:
\begin{equation*} \label{eq:directcalc}
|\frac{d}{d t} 
  \lambda _{t} ( \widetilde{\rho} (\tau _{0} ) )
({\rho _{3} (\tau _{0}) }) |
 \leq \max _{t \in [0,1], M} |\frac{df _{t}}{dt}| \cdot
\int _{\widetilde{\rho}  (\tau _{0} )} \lambda. 
\end{equation*}   

On the other hand $$
\int _{\widetilde{\rho}  (\tau _{0} )} \lambda 
 \leq \max _{M,t} f _{t} \cdot \Lambda  
 (\widetilde{\rho}  (\tau _{0} )). $$ 

Consequently:
\begin{equation*} \label{eq:growthbound}
\frac{d}{d \tau} \rho _{2} (\tau _{0} ) = \frac{d}{d\tau}  \left(\Lambda \circ \widetilde{\rho}  (\tau)\right) (\tau _{0} )
\leq \max _{t \in [0,1], M} |\frac{df
   _{t}}{dt}| \cdot
     \max _{t,M} f _{t} \cdot \Lambda (
 \widetilde{\rho}  (\tau _{0} )) \cdot L = \max _{t \in [0,1], M} |\frac{df
   _{t}}{dt}| \cdot
     \max _{t,M} f _{t} \cdot L \cdot \rho _{2} (\tau _{0} ), 
\end{equation*}
for almost all $\tau _{0} \in [0,1] $.

\end{proof}
\qed
\section{Acknowledgements} I am grateful to Viktor Ginzburg,  Dusa McDuff,  Helmut Hofer, John Pardon, and Dennis Sullivan for
comments and interest. Viktor Ginzburg in particular for introducing me to the Fuller index.
 \bibliographystyle{siam}  
   \bibliography{/home/yasha/texmf/bibtex/bib/link} 
\end {document}